\documentclass{birkmult}
 \newtheorem{thm}{Theorem}[section]
 \newtheorem{cor}[thm]{Corollary}
 \newtheorem{lem}[thm]{Lemma}
 \newtheorem{prop}[thm]{Proposition}
 \theoremstyle{definition}
 \newtheorem{defn}[thm]{Definition}
 \theoremstyle{remark}
 \newtheorem{rem}[thm]{Remark}

 \numberwithin{equation}{section}
\def\opn#1#2{\def#1{\operatorname{#2}}}
\opn\depth{depth}
\opn\reg{reg}
\opn\height{height}
\opn\pd{pd}
\opn\Tor{Tor}
\opn\Ass{Ass}
\opn\multideg{multideg}

\def\NZQ{\mathbb}               

\def\CC{{\NZQ C}}

\def\opn#1#2{\def#1{\operatorname{#2}}}
\opn\depth{depth}
\opn\dim{dim}
\opn\link{link}
\opn\type{type}
\opn\reg{reg}
\opn\height{height}
\opn\pd{pd}
\opn\Tor{Tor}
\opn\Ass{Ass}
\opn\multideg{multideg}

\begin{document}

%
%
%
%
%
%
%
%
%

\title[Betti numbers of mixed product ideals]
 {Betti numbers of mixed product ideals}
\author{Giancarlo Rinaldo}
\address{Dipartimento di Matematica\br 
Universit\`a di Messina\br
Contrada Papardo, salita Sperone, 31\br
98166 Messina, Italy}
\email{rinaldo@dipmat.unime.it}
\subjclass{Primary 13P10; Secondary 13B25}

\keywords{Mixed product ideals, Betti numbers, Cohen-Macaulay rings}

\begin{abstract}
We compute the Betti numbers of the resolution of a special class of square-free monomial ideals, the ideals of mixed products.  Moreover when these ideals are Cohen-Macaulay we calculate their type.
\end{abstract}

\maketitle

\section{Introduction}
\label{1}

 The class of ideals of mixed products is a special class of square-free monomial ideals. They were first introduced by G.~Restuccia and R.~Villarreal (see \cite{RV} and \cite{Vi}), who studied the normality of such ideals.

In \cite{IR} C.~Ionescu and G.~Rinaldo studied the Castelnuovo-Mumford regularity, the depth and dimension of mixed product ideals and characterize when they are Cohen-Macaulay.

Let $S=K[\mathbf{x},\mathbf{y}]$ be a polynomial ring over a field $K$ in two disjoint
sets of variables $\mathbf{x}=(x_1,\dots,x_n)$, $\mathbf{y}=(y_1,\dots,y_m)$. The \textit{ideals of mixed products} are the ideals 
$$I_qJ_r+I_sJ_t,\ \ q,r,s,t\in{\bf N},\ \ q+r=s+t,$$
where $I_l$ (resp. $J_p$) is the ideal of $S$ generated by all the
square-free monomials of degree $l$ (resp. $p$) in the variables $\mathbf{x}$
(resp. $\mathbf{y}$). We set  $I_0=J_0=S$.
By symmetry, essentially there are  2 cases:
\begin{enumerate}
 \item[i)] $L=I_kJ_r+I_sJ_t,\ 0\leq k< s$;
 \item[ii)] $L=I_kJ_r,\ k\geq 1\ {\rm or}\ r\geq 1.$
\end{enumerate}

\noindent
Let $\mathbb{F}$ be the minimal free resolution of an ideal $I$:
\[\mathbb{F}:\  0\to F_p \to \cdots \to F_{p-1}\to\cdots\to F_0\to I\rightarrow 0\]

\noindent
and let $F_i=\mathop\oplus\limits_jS(-j)^{\beta_{ij}}.$

We want to calculate the Betti numbers, $\beta_{ij}$, when $I$ is a mixed product ideals.

To reach this goal we consider $S/I$ as a Stanley-Reisner ring (see \cite{BH}, Chapter 5), we calculate the Alexander duality of $I$ and we use the Hochster's formula (see \cite{BH}, Chapter 5.5 or \cite{MS}, Chapter 1.5).

In section \ref{sec:Alex} we calculate the dual of the mixed product ideal $I_q$ and $I_qJ_r$.

Hochster's formula give us a powerful tool to calculate Betti numbers of the resolution of a squarefree monomial ideal $I$. For this reason we have to calculate $I^*$ and the homology of the chain complex of the simplicial complex $\Delta^*$ related to $I^*$ (section \ref{sec:Betti}).

In the last section as an application we compute the type of Cohen-Macaulay mixed product ideals. 

\section{Alexander duality of mixed product ideals}\label{sec:Alex}
We recall the following (see \cite{MS}, Definition 1.35):
\begin{defn}
  Let $I=(\mathbf{x}^{\alpha_1},\ldots, \mathbf{x}^{\alpha_q})\subset K[\mathbf{x}]= K[x_1,\ldots,x_n]$ be a square-free
  monomial ideal, with $\alpha_i=(\alpha_{i_1},\ldots,\alpha_{i_n})\in \{0,1\}^n$. The \emph{Alexander dual}
  of $I$ is the ideal
  
  $$I^*=\mathop\bigcap\limits_{i=1}^q \mathfrak{m}_{\alpha_i},$$  
where $\mathfrak{m}_{\alpha_i}=(x_j:\alpha_{i_j}=1)$.
\end{defn}

\begin{prop}\label{theo:dualq}
Let $S=K[x_1,\ldots, x_n]$, $I_q\subset S$ with $1\leq q\leq n$. Then
  $$I_q^*=I_{n-q+1}.$$ 
\end{prop}

\begin{proof}
Let $I_q=(x_{i_1}\cdots x_{i_q}: 1\leq i_1 < \cdots < i_q \leq n)$, we consider $I_q$ as a Stanley-Reisner ideal of the simplicial complex $\Delta$, that is
$$I_q=I_{\Delta}=(x^\tau :\tau\notin \Delta),$$
where $\tau=\{i_1,\ldots,i_q\}$ and $x^\tau=x_{i_1}\cdots x_{i_q}$.

This implies that the simplicial complex $\Delta$ is pure and its facets are all the facets of dimension $q-2$, that are $\sigma=\{i_1,\ldots,i_{q-1}\}$ with $1\leq i_1<i_2<\ldots<i_{q-1}\leq n$.

We want to calculate $\Delta^*=\{\overline{\tau}:\tau\notin \Delta\}$. If we consider a simplex $\sigma \in \Delta$ such that $|\sigma|<q-1$ then $\sigma \cup \{i\}\in \Delta$. Therefore to obtain a minimal non face of $\Delta$ we have to consider a facet in $\Delta$ and add a vertex not in the facet. That is $\tau=\{i_1, \ldots,i_q\}$, with $1\leq i_1<i_2<\ldots<i_{q}\leq n$.

Therefore $\overline{\tau}=\{i'_1,\ldots, i'_{n-q}\}$, with $1\leq i'_1< \cdots <i'_{n-q}\leq n$, are the facets of $\Delta^*$. The minimal non faces of $\Delta^*$ are $\{i'_1,\ldots,i'_{n-q+1}\}$ and its Stanley-Reisner ideal is
$$I_{\Delta^*}=(x_{i'_1}\cdots x_{i'_{n-q+1}}:1\leq i'_1< \cdots <i'_{n-q+1}\leq n )=I_{n-q+1}.$$
\end{proof}

\begin{prop}\label{theo:dualsum}
Let $S=K[x_1,\ldots,x_n, y_1,\ldots,y_m]$, $(I_q+J_r)\subset S$ with $1\leq q \leq n$, $1\leq r \leq m$. Then
  $$(I_q+J_r)^*=I_{n-q+1} J_{m-r+1}.$$ 
\end{prop}

\begin{proof}
Let $$I_q+J_r=(x_{i_1}\cdots x_{i_q}, y_{j_1}\cdots y_{j_r}: 1\leq i_1 < \cdots < i_q \leq n,1\leq j_1 < \cdots < j_r \leq m),$$
by Alexander duality we have 

 $$(I_q+J_r)^*=(\mathop\bigcap\limits_\alpha \mathfrak{m}_\alpha) \cap (\mathop\bigcap\limits_\beta \mathfrak{n}_\beta),$$
where $\mathfrak{m}_\alpha=(x_{i_1},\ldots,x_{i_q})$, $\mathfrak{n}_\beta=(y_{j_1},\ldots,y_{j_r})$.
By the coprimality of $\mathfrak{m}$ and $\mathfrak{n}$ we obtain
 $$(\mathop\bigcap\limits_\alpha \mathfrak{m}_\alpha )\cap (\mathop\bigcap\limits_\beta \mathfrak{n}_\beta)=(\mathop\bigcap\limits_\alpha\mathfrak{m}_\alpha ) (\mathop\bigcap\limits_\beta \mathfrak{n}_\beta)=(I_q)^* (J_r)^*,$$
and by Proposition \ref{theo:dualq} we have the assertion.

\end{proof}

\begin{cor}\label{cor:dualprod}
Let $S=K[x_1,\ldots,x_n,y_1,\ldots,y_m]$, $I_qJ_r\subset S$ with $1\leq q \leq n$, $1\leq r \leq m$. Then
  $$(I_qJ_r)^*=I_{n-q+1}+ J_{m-r+1}.$$ 
\end{cor}
\begin{proof}
 By Alexander duality $(I^*)^*=I$, therefore the assertion follows by Proposition \ref{theo:dualsum}.
\end{proof}

\begin{rem}\label{rem:newmix}
 We observe that in general the class of mixed product ideals is not closed under Alexander duality. In fact 
\[(I_qJ_r+ I_sJ_t)^*=I_{n-q+1} + I_{n-s+1} J_{m-r+1} +J_{m-t+1},\]
with $1\leq q< s \leq n,\ 1\leq t< r\leq m$. 

To calculate this duality we observe that
 $$(I_qJ_r+I_sJ_t)^*=(\mathop\bigcap\limits_{\alpha,\beta} (\mathfrak{m}_\alpha +\mathfrak{n}_\beta ))\cap (\mathop\bigcap\limits_{\gamma,\delta} (\mathfrak{m}_\gamma+\mathfrak{n}_\delta)),$$
where $\mathfrak{m}_\alpha=(x_{i_1},\ldots,x_{i_q})$, $\mathfrak{n}_\beta=(y_{j_1},\ldots,y_{j_r})$, $\mathfrak{m}_\gamma=(x_{i'_1},\ldots,x_{i'_s})$
and  $\mathfrak{n}_\delta=(y_{j'_1},\ldots,y_{j'_t})$.

By Corollary \ref{cor:dualprod} we obtain
$$(\mathop\bigcap\limits_{\alpha,\beta} \mathfrak{m}_\alpha +\mathfrak{n}_\beta )\cap (\mathop\bigcap\limits_{\gamma,\delta} \mathfrak{m}_\gamma+\mathfrak{n}_\delta)=
(I_{n-q+1} + J_{m-r+1})\cap (I_{n-s+1} + J_{m-t+1}),$$
and since $I_{n-s+1}\supset I_{n-q+1}$,  $J_{m-r+1}\supset J_{m-t+1}$ we have the assertion.
\end{rem}

%

\section{Betti numbers of mixed product ideals}\label{sec:Betti}
Let $K$ be a field, $S=K[x_1,\ldots,x_n]$ be a polynomial ring, $I\subset S$, and let $\mathbb{F}$ be the minimal free resolution of the ideal $I$.
Then
\[\mathbb{F}:\  0\to F_p \to \cdots \to F_{p-1}\to\cdots\to F_0\to I\rightarrow 0\]

\noindent
with $F_i=\mathop\oplus\limits_jS(-j)^{\beta_{ij}}$. We want to calculate the Betti numbers, $\beta_{ij}$, when $I$ is a mixed product ideals.

\begin{lem}\label{th:bettiq}
Let $S=K[x_1,\ldots,x_n]$, $I_q\subset S$ with $1\leq q\leq n$.
 $$\beta_i (I_q)=\beta_{i,q+i} (I_q)= \binom {n} {q+i} \binom {q+i-1} {i},\hspace{.5cm}i\geq 0.$$
\end{lem}
\begin{proof}
We observe that $I_q$ has a $q$-linear resolution (see Example 2.2, \cite{HH})  therefore the Betti number of each free-module in the FFR of $I_q$ is
$$\beta_i (I_q) =\beta_{i,q+i} (I_q).$$

 Let $I_\Delta=I_q$, where $\Delta$ is the simplicial complex defined in the proof of Proposition  \ref{theo:dualq}. By Hochster's formula (see Corollary 1.40, \cite{MS}) we have
$$ \beta_{i,q+i} (I_\Delta)= \sum_{|\sigma|=q+i} \beta_{i,\sigma}(I_\Delta) =\sum_{\overline{\sigma}} \dim_K \widetilde{H}_{i-1}(\link_{\Delta^*}(\overline{\sigma};K).$$

By the symmetry between simplices in $\Delta$ of the same dimension, we obtain
$$\sum_{\overline{\sigma}} \dim_K \widetilde{H}_{i-1}(\link_{\Delta^*}(\overline{\sigma};K)=\binom {n} {q+i} \dim_K \widetilde{H}_{i-1}(\link_{\Delta^*}(\overline{\sigma}_0;K) $$
where $\sigma_0=\{1,\ldots, q+i\}\in \Delta$, $\overline{\sigma}_0=\{q+i+1,\ldots,n\}\in \Delta^*$ and $\binom {n} {q+i}$ is the number of simplices of dimension $q+i-1$ in $\Delta$.

Since $I_{\Delta^*}=I_{n-q+1}$ (see Proposition \ref{theo:dualq}) the facets of $\Delta^*$ are $\{{i_1},\ldots, {i_{n-q}}\}$, $1\leq i_1 < \ldots < i_{n-q}\leq n$, therefore 
$$\link \overline{\sigma}_0=\{\{1,\ldots,i\},\ldots, \{q+1,\ldots,q+i\} \}.$$

The chain complex of $\link  \overline{\sigma}_0$ is
$$0\to C_{i-1} \stackrel{\partial_{i-1}}{\to} \cdots \stackrel{\partial_1}{\to} C_0 \stackrel{\partial_0}{\to} C_{-1} \to 0, $$
where $\dim_K C_j =\binom {q+i}{j+1}$, $j=-1,0,\ldots, i-1$.

It is easy to observe that $\widetilde{H}_{j}(\link_{\Delta^*}(\overline{\sigma}_0;K))=0$, for $j=0,\ldots,i-2$, since it is the ``truncated'' chain complex of the simplex with $q+i$ vertices, $\sigma_0$. Therefore we need to calculate $\dim_K (\widetilde{H}_{i-1}(\link_{\Delta^*}(\overline{\sigma}_0;K))\cong \ker \partial_{i-1})$.

We want to show that $\dim_K \ker \partial_{j-1}=\binom {q+i-1}{j}$ and we make induction on the length of the exact sequence of vector spaces, $j$,
$$0\to \ker \partial_{j-1} \to   C_{j-1} \stackrel{\partial_{j-1}}{\to} \cdots \stackrel{\partial_1}{\to} C_0 \stackrel{\partial_0}{\to} C_{-1} \to 0. $$

For $j=1$ we obtain

$$0\to \ker \partial_{0} \to   C_{0} \stackrel{\partial_0}{\to} C_{-1} \to 0 $$
and 
$$\dim_K \ker \partial_{0}= \dim_K C_{0} - \dim_K C_{-1}= q+i -1$$
as expected.

Let $\dim_K \ker \partial_{j-1}=\binom {q+i-1}{j}$. We consider the short exact sequence

$$0\to \ker \partial_{j} \to   C_{j} \stackrel{\partial_{j}}{\to} \ker \partial_{j-1}\to 0,$$
and since
$$\binom{q+i-1}{j+1}+\binom{q+i-1}{j}= \binom{q+i}{j+1}$$
we obtain the assertion.
\end{proof}

\begin{thm}\label{theo:bettiqr}
Let $S=K[x_1,\ldots,x_n,y_1,\ldots,y_m]$, $I_qJ_r\subset S$ with $1\leq q \leq n$, $1\leq r \leq m$. Then for $i\geq 0$, we have
\begin{eqnarray*}
 \beta_i (I_qJ_r) & = & \beta_{i,q+r+i} (I_qJ_r)\\
                  & = & \sum_{j+k=i} \binom {n} {q+j} \binom {m} {r+k}\binom {q+j-1} {j}\binom {r+k-1} {k}.
\end{eqnarray*}
\end{thm}
\begin{proof}
Since $I_qJ_r$ has a $(q+r)$-linear resolution (see Lemma 2.5, \cite{IR}) the Betti number of each free-module in the FFR of $I_qJ_r$ is
$$\beta_i (I_qJ_r) =\beta_{i,q+r+i} (I_qJ_r).$$
Let $I_\Delta=I_qJ_r$, by Hochster's formula we have that
$$ \beta_{i,q+r+i} (I_\Delta)= \sum_{|\sigma|=q+r+i} \beta_{i,\sigma}(I_\Delta) =\sum_{\overline{\sigma}} \dim_K \widetilde{H}_{i-1}(\link_{\Delta^*}(\overline{\sigma};K).$$

Let $\deg$ a bidegree on the ring $S=K[x_1,\ldots,x_n,y_1,\ldots, y_m]$, where $\deg(x_i)=(1,0)$ and $\deg(y_j)=(0,1)$.

It is easy to observe that the number of squarefree monomials in $S$ of the same bidegree $(q+j,r+k)$ is
$$
\binom {n} {q+j} \binom {m} {r+k}.
$$

By the symmetry between simplices with the ``same bidegree'' in $\Delta$, we fix  $\sigma_0=\{x_1,\ldots, x_{q+j}, y_1,\ldots, y_{r+k}\}$
and its corresponding Alexander dual $\overline{\sigma}_0=\{ x_{q+j+1},\ldots, x_n, y_{r+k+1},\ldots,y_m\}\in \Delta^*$ and calculate  

$$ \dim_K \widetilde{H}_{i-1}(\link_{\Delta^*}(\overline{\sigma}_0;K), $$
with $i=j+k$.

By Corollary \ref{cor:dualprod}, $I_{\Delta^*}=I_{n-q+1}+J_{m-r+1}$, and the facets of 
$\Delta^*$ are 
\[\{x_{i_1},\ldots, x_{i_{n-q}}, y_{j_1},\ldots,y_{j_{m-r}} \},\]
$1\leq i_1 < \ldots < i_{n-q}\leq n$, $1\leq j_1 < \ldots < j_{m-r}\leq m$, therefore 
$$\link \overline{\sigma}_0=\{\{x_1,\ldots,x_j,y_1,\ldots,y_k\},\ldots, \{x_{q+1},\ldots,x_{q+j}, y_{r+1},\ldots, y_{r+k}\}\}.$$
Let $\sigma_x=\{\{x_1,\ldots,x_j\},\ldots, \{x_{q+1},\ldots,x_{q+j}\}\}$ and let
\[
\CC_x: 0\to C_{j-1} \stackrel{\partial_{j-1}}{\to} \cdots \stackrel{\partial_1}{\to} C_0 \stackrel{\partial_0}{\to} C_{-1} \to 0,
\]
the chain complex of $\sigma_x$, let $\sigma_y=\{\{y_1,\ldots,y_k\},\ldots, \{y_{r+1},\ldots, y_{r+k}\}\}$ and let
\[
\CC_y: 0\to C_{k-1} \stackrel{\partial_{k-1}}{\to} \cdots \stackrel{\partial_1}{\to} C_0 \stackrel{\partial_0}{\to} C_{-1} \to 0,
\]
the chain complex of $\sigma_y$.

It is easy to see that the tensor product of the two complexes, $\CC_x\otimes\CC_y$ , is the chain complex of $\link \overline{\sigma}_0$.

Therefore $\widetilde{H}_{i-1}(\link_{\Delta^*}(\overline{\sigma}_0;K)$, with $i=j+k$, is isomorphic to $\ker \partial_{j-1}\otimes \ker \partial_{k-1}$, and from Lemma \ref{th:bettiq} the assertion follows.

\end{proof}

\begin{thm}\label{theo:bettimix}
Let $S=K[x_1,\ldots,x_n,y_1,\ldots,y_m]$, $(I_qJ_r+I_sJ_t)\subset S$ with $0\leq q < s \leq n $, $0 \leq t < r\leq m$. Then
 $$\beta_i (I_qJ_r+I_sJ_t) = \beta_i (I_qJ_r) + \beta_i (I_sJ_t) + \beta_{i-1} (I_sJ_r),\hspace{.5cm}i\geq 1.$$ 
\end{thm}

\begin{proof}

We consider the exact sequence
  \begin{equation}
    \label{eq:exact1}
    0\to S/I_s J_r\to S/I_qJ_r\oplus S/I_sJ_t\to S/(I_qJ_r+I_sJ_t )\to 0.
  \end{equation}
  
We want to study $\Tor_i(K,S/(I_qJ_r+I_sJ_t))_{i+j}$. By the sequence (\ref{eq:exact1}) we obtain the long exact sequence of $\Tor$ 
  \begin{eqnarray*}
    \cdots\to\Tor_i(K,S/I_s J_r)_{i+j}&\to&\Tor_i(K,S/I_qJ_r\oplus S/I_sJ_t)_{i+j} \to\\
    \to \Tor_i(K,S/(I_qJ_r+I_sJ_t))_{i+j} &\to&  \Tor_{i-1}(K,S/I_sJ_r)_{i+j}  \to\cdots\\
    \to \Tor_{i-1}(K,S/I_qJ_r \oplus S/I_sJ_t)_{i+j} &\to&  \Tor_{i-1}(K,S/I_sJ_r)_{i+j}  \to\cdots
  \end{eqnarray*}
  
We observe that $S/(I_kJ_l)$ has a $(k+l-1)$-linear resolution therefore the Betti number of each free-module in the FFR of $S/I_kJ_l$ is
$$\beta_i (S/I_kJ_l) =\beta_{i,k+l+i-1} (S/I_kJ_l).$$

This implies that if $j\notin \{q+r-1,s+t-1,s+r-2\}$, $\Tor_i(K,S/I_qJ_r\oplus S/I_sJ_t)_{i+j}=0$ and $\Tor_{i-1}(K,S/I_sJ_r)_{i+j}=0$, that is 
\[\Tor_i(K,S/(I_qJ_r+I_sJ_t))_{i+j}=0.\]
Therefore we have to study the following cases: 
\begin{enumerate}
 \item [1)] $j=s+r-2$;
 \item [2)] $j=q+r-1$;
 \item [3)] $j=s+t-1$.	
\end{enumerate}
We may assume $q+r\geq s+t$.

1) Let $j=s+r-2$. Then $\Tor_i(K,S/I_s J_r)_{i+j}=0$ and $\Tor_{i-1}(K,S/I_qJ_r \oplus S/I_sJ_t)_{i+j}=0$. Therefore we have the exact sequence

\medskip
\noindent
$0\to \Tor_i(K,S/I_qJ_r\oplus S/I_sJ_t)_{i+j}\to$ 

$\hfill\to \Tor_i(K,S/(I_qJ_r+I_sJ_t))_{i+j} \to  \Tor_{i-1}(K,S/I_sJ_r)_{i+j}\to 0.$
\medskip

If $s+r-2>q+r-1$ we obtain  
\[\Tor_i(K,S/(I_qJ_r+I_sJ_t))_{i+j} \cong  \Tor_{i-1}(K,S/I_sJ_r)_{i+j},\]
and to finish we have to study degree 2) and 3).

If $s+r-2=q+r-1>s+t+1$, that is degree 1) and 2) coincide, we obtain  

\medskip
\noindent
$\dim_K \Tor_i (K,S/(I_qJ_r+I_sJ_t))_{i+j} =$

$\hfill\dim_K \Tor_{i}(K,S/I_qJ_r)_{i+j}+\dim_K \Tor_{i-1}(K,S/I_sJ_r)_{i+j}, $
\medskip

\noindent
and we have to study degree 3) to finish.

If $s+r-2=q+r-1=s+t+1$, that is cases 1), 2) and 3) coincide,

\medskip
\noindent
$\dim_K \Tor_i (K,S/(I_qJ_r+I_sJ_t))_{i+j} =\dim_K \Tor_{i}(K,S/I_qJ_r)_{i+j}+$

$\hfill+\dim_K \Tor_{i}(K,S/I_sJ_t)_{i+j}+\dim_K \Tor_{i-1}(K,S/I_sJ_r)_{i+j},$

\medskip
\noindent
and this case is complete.

2) The case $j=q+r-1=s+r-2$ has been already studied in case 1), therefore we assume $j=q+r-1<s+r-2$. We obtain the exact sequence
\[0\to \Tor_i(K,S/I_qJ_r\oplus S/I_sJ_t)_{i+j}\to \Tor_i(K,S/(I_qJ_r+I_sJ_t))_{i+j} \to 0\]
and we continue as in case 1).

3) This case is similar to 2).

The assertion follows.
\end{proof}
\begin{rem}
 For completness we compute the Betti numbers of $(I_s+I_qJ_t+J_r)$ (see Remark \ref{rem:newmix}). That is
 $$\beta_i (I_s+I_qJ_t+J_r) = \beta_i (I_s) + \beta_i (I_qJ_t) +\beta_i (J_r) + \beta_{i-1} (I_qJ_r) + \beta_{i-1} (I_sJ_t),\hspace{.5cm}i\geq 1.$$
To prove this is enough to observe that $(I_s)\cap (I_qJ_t+J_r)=I_sJ_t+I_sJ_r=I_sJ_t$ and we have the exact sequence
  \begin{equation}
    \label{eq:exact2}
    0\to S/I_sJ_t\to S/I_s\oplus S/(I_qJ_t+J_r)\to S/(I_s+I_qJ_t+J_r )\to 0.
  \end{equation}
To obtain the assertion we continue as in the proof of Theorem \ref{theo:bettimix}.
\end{rem}

\section{Cohen-Macaulay Type of mixed product ideals}
Let $K$ be a field, $S=K[x_1,\ldots,x_n]$ be a polynomial ring, $I\subset S$ a graded ideal. We consider $S/I$ as a standard graded $K$-algebra.

We have the following (see \cite{Vi}, Proposition 4.3.4):
\begin{prop}
 Let $S/I$ be a Cohen-Macaulay ring then the type of $S/I$ is equal to the last Betti number in the minimal free  resolution of $S/I$ as an $S$-module.
\end{prop}

In \cite{IR} we classified the Cohen-Macaulay algebra $S/I$ when $I$ is a mixed product ideal.
From now on we say CM for short.

\begin{prop}\label{CM} Let $S=K[x_1,\ldots,x_n,y_1,\ldots,y_m]$.
\begin{enumerate}
\item Let $0<q\leq n$ then $S/I_q$ is CM with $\dim S/I_q=m+q-1$;
\item Let $0<q\leq n$ and $0<r\leq m$ then $S/(I_qJ_r)$ is CM if and only if $r=m$ and $q=n$ with $\dim S/I_nJ_m=m+n-1$;
\item Let $0<q<s \leq n$ and $0<r\leq m$ then $S/(I_qJ_r+I_s)$ is CM if and only if $s=q+1$ and $r=m$ with $\dim S/(I_qJ_m+I_{q+1})=m+q-1$;
\item Let $0<q<s \leq n$ and $0<t<r\leq m$ then $S/(I_qJ_r+I_sJ_t)$ is CM if and only if $r=m$, $s=n$, $t=m-1$, $q=n-1$ with $\dim S/(I_{n-1}J_m+I_nJ_{m-1})=m+n-2.$
\end{enumerate}
\end{prop}
\begin{proof}
 See Proposition 3.1, Theorem 3.2 and Corollary 3.8 of \cite{IR}, and observe  that a CM ideal is unmixed. 
\end{proof}

By the results of Section \ref{sec:Betti} we want to calculate the type of the CM algebras classified in Proposition \ref{CM}.

\begin{prop} Let $S=K[x_1,\ldots,x_n]$, $I_q\subset S$ and $1\leq q\leq n$ then 
\[\type S/I_q= \binom{n-1}{n-q}\]
\end{prop}
\begin{proof}
 By Auslander-Buchsbaum Theorem (Theorem 1.3.3, \cite{BH}) and Proposition \ref{CM}.(1), we have 
\[\pd S/I_q= \dim S- \depth S/I_q= \dim S- \dim S/I_q= n-q+1.\] 
By Lemma \ref{th:bettiq} and observing that $\beta_i (I_q)=\beta_{i+1}(S/I_q)$ we have 
\[\type(S/I_q)=\beta_{n-q+1}(S/I_q)=\binom{n-1}{n-q}.\]
\end{proof}

From now on let $S=K[x_1,\ldots,x_n,y_1,\ldots,y_m]$.
\begin{prop}
Let $S/I_qJ_r$ be a Cohen-Macaulay algebra then
\[\type S/I_qJ_r= 1\]
with $q=n$, $r=m$.
\end{prop}
\begin{proof}
 By Corollary \ref{CM}.(2), we have that $S/I_qJ_r$ is CM if and only if $I_qJ_r$ is a principal ideal and the assertion follows.
\end{proof}

\begin{prop}
Let $S/(I_qJ_r+I_s)$ be a Cohen-Macaulay algebra then
\[\type S/(I_qJ_r+I_s)= \binom{n-1}{n-q} + \binom{n-1}{n-q-1}.\]
with $s=q+1$ and $r=m$.
\end{prop}
\begin{proof}
 By Proposition \ref{CM}.(3), we have that $S/(I_qJ_r+I_s)$ is CM if and only if $s=q+1$ and $r=m$ therefore by Theorem 3.2.(2) of \cite{IR} 
\[\dim S/(I_qJ_m+I_{q+1})=m+q-1.\]
By Auslander-Buchsbaum Theorem we obtain
\[\pd S/(I_qJ_m+I_{q+1})= \dim S-\dim S/(I_qJ_m+I_{q+1})= n-q+1.\] 
Applying Theorem \ref{theo:bettimix}, with $i=n-q$, $s=q+1$, $r=m$,  we have
\[\beta_{n-q+1}(S/(I_qJ_m+I_{q+1})) = \beta_{n-q} (I_qJ_m) + \beta_{n-q} (I_{q+1}) + \beta_{n-q-1} (I_{q+1}J_m).\]
Since the $\pd S/I_{q+1}=n-q$, we have that $\pd I_{q+1}=n-q-1$, that is $\beta_{n-q} (I_{q+1})=0$.

Therefore 
\[\beta_{n-q+1}(S/(I_qJ_m+I_{q+1})) = \beta_{n-q} (I_qJ_m) + \beta_{n-q-1} (I_{q+1}J_m),\]
and by Theorem \ref{theo:bettiqr}
\[\beta_{n-q} (I_qJ_m)= \sum_{j+k=n-q} \binom {n} {q+j} \binom {m} {m+k}\binom {q+j-1} {j}\binom {m+k-1} {k},\]
\[\beta_{n-q-1} (I_{q+1}J_m)= \sum_{j+k=n-q-1} \binom {n} {q+1+j} \binom {m} {m+k}\binom {q+j} {j}\binom {m+k-1} {k}.\]
It is easy to see that for $k>0$ the summands in $\beta_{n-q} (I_qJ_m)$ and in $\beta_{n-q-1} (I_{q+1}J_m)$ vanish.

If $k=0$ we obtain the assertion.
\end{proof}

\begin{prop}
Let $S/(I_qJ_r+I_sJ_t)$ be a Cohen-Macaulay algebra then
\[\type S/(I_qJ_r+I_sJ_t)= m+n-1,\]
with $q=n-1$, $r=m$, $s=n$, $t=m-1$.
\end{prop}
\begin{proof}
 By Proposition \ref{CM}.(4), we have that $S/(I_qJ_r+I_sJ_t)$ is CM if and only if $q=n-1$, $r=m$, $s=n$, $t=m-1$, therefore by Theorem 3.2.(2) \cite{IR} 
\[\dim S/(I_{n-1}J_m+I_nJ_{m-1})=m+n-2.\]
By Auslander-Buchsbaum Theorem we obtain
\[\pd S/(I_{n-1}J_m+I_nJ_{m-1})= \dim S-\dim S/(I_{n-1}J_m+I_nJ_{m-1})= 2.\] 

Applying Theorem \ref{theo:bettimix}, with $i=1$, $q=n-1$, $r=m$, $s=n$, $t=m-1$,  we obtain
\[\beta_{2}(S/(I_{n-1}J_m+I_nJ_{m-1})) = \beta_1 (I_{n-1}J_m) + \beta_1 (I_n J_{m-1}) + \beta_0 (I_nJ_m).\]

We observe that $I_nJ_m$ is a principal ideal therefore $\beta_0 (I_nJ_m)=1$.

By Theorem \ref{theo:bettiqr}
\[\beta_1 (I_{n-1}J_m)= \sum_{j+k=1} \binom {n} {n-1+j} \binom {m} {m+k}\binom {n-1+j-1} {j}\binom {m+k-1} {k},\]

\[\beta_1 (I_nJ_{m-1})= \sum_{j+k=1} \binom {n} {n+j} \binom {m} {m-1+k}\binom {n+j-1} {j}\binom {m-1+k-1} {k}.\]
It is easy to see that for $k>0$ the summands in $\beta_{1} (I_{n-1}J_m)$  vanish. 
For $j>0$ the same happens to  $\beta_1 (I_n J_{m-1})$.

Therefore we obtain
\[\beta_{2}(S/(I_{n-1}J_m+I_nJ_{m-1})) = n-1 + m-1 + 1.\]
\end{proof}

\begin{rem}
 We observe that the mixed product ideal that are Gorenstein $($e.g. CM ring of type 1$)$, are:
\begin{enumerate}
 \item $K[x_1,\ldots,x_n]/I_n$;
 \item $K[x_1,\ldots,x_n]/(x_1,\ldots,x_n)$;
 \item $K[x_1,\ldots,x_n,y_1,\ldots,y_m]/I_nJ_m$.
\end{enumerate}
\end{rem}

\end{document}